\documentclass{article}
\usepackage{amsmath, amsthm, amscd, amsfonts, amssymb, graphicx, color}
\usepackage[bookmarksnumbered, colorlinks, plainpages]{hyperref}

\newtheorem{theorem}{Theorem}[section]
\newtheorem{lemma}[theorem]{Lemma}
\newtheorem{proposition}[theorem]{Proposition}

\theoremstyle{definition}
\newtheorem{definition}[theorem]{Definition}

\theoremstyle{remark}
\newtheorem{remark}[theorem]{Remark}
\numberwithin{equation}{section}

\input{mathrsfs.sty}

\newcommand{\ket}[1]{ | #1  \rangle}
\newcommand{\bra}[1]{ \langle #1  |}
\newcommand{\braket}[2]{ \left\langle #1  | #2  \right\rangle}
\newcommand{\ii}{\mathbf{i_1}}
\newcommand{\iii}{\mathbf{i_2}}
\newcommand{\jj}{\mathbf{j}}
\newcommand{\ee}{\mathbf{e_1}}
\newcommand{\eee}{\mathbf{e_2}}
\newcommand{\e}[1]{\mathbf{e_{#1}}}
\newcommand{\ik}[1]{\mathbf{i_{#1}}}
\newcommand{\T}{\mathbb{T}}
\newcommand{\M}{\mathbb{M}}
\newcommand{\C}{\mathbb{C}}
\newcommand{\D}{\mathbb{D}}
\newcommand{\R}{\mathbb{R}}
\newcommand{\nc}{\mathcal{NC}}
\renewcommand{\(}{\left(}
\renewcommand{\)}{\right)}

\renewcommand{\P}[2]{P_{#1}( #2 )}

\newcommand{\bo} {\ensuremath{{\bf i_1}}}

\newcommand{\eo} {\ensuremath{{\bf e_1}}}
\newcommand{\et} {\ensuremath{{\bf e_2}}}
\newcommand{\ek} {\ensuremath{{\bf e_k}}}
\newcommand{\iic}{{\rm \bf{i}}_{\bf{1}}^2}
\newcommand{\iiic}{{\rm \bf{i}}_{\bf{2}}^2}

\newcommand{\hh}{{\widehat{1}}}
\newcommand{\hhh}{{\widehat{2}}}
\newcommand{\mC}{\ensuremath{\mathbb{C}}}

\newcommand{\mN}{\ensuremath{\mathbb{N}}}
\newcommand{\mR}{\ensuremath{\mathbb{R}}}

\newcommand{\mo}{\mathbf{1}}
\newcommand{\mt}{\mathbf{2}}
\newcommand{\mk}{\mathbf{k}}
\newcommand{\scalarmath}[2]{\( #1, #2 \)}

\begin{document}
\vspace{4cm}
\begin{center} \LARGE{\textbf{Bicomplex Riesz-Fischer Theorem}}
\end{center}
\vspace{1cm}

\begin{center} \bf{K. S. Charak$^{1 }$,\quad R. Kumar$^{2}$,\quad D. Rochon$^{3}$ }
\end{center}

\bigskip

\begin{center}
{$^{1}$ Department of Mathematics, University of Jammu,\\
Jammu-180 006, INDIA.\\
E-mail: kscharak7@rediffmail.com }
\end{center}
\medskip

\begin{center}{$^{2}$ Department of Mathematics, University of Jammu,\\
Jammu-180 006, INDIA.\\
E-mail: ravinder.kumarji@gmail.com }
\end{center}

\medskip
\begin{center} {$^{3}$ D\'epartement de math\'ematiques et d'informatique,\\
Universit\'e du Qu\'ebec \`a Trois-Rivi\`eres, C.P. 500, Trois-Rivi\`eres, Qu\'ebec, Canada G9A 5H7.  \\
E-mail: Dominic.Rochon@UQTR.CA,\\
Web: www.3dfractals.com}
\end{center}

\begin{abstract}
\noindent This paper continues the study of infinite dimensional bicomplex
Hilbert spaces introduced in previous articles on the topic. Besides obtaining a Best Approximation
Theorem, the main purpose of this paper is to obtain a bicomplex
analogue of the Riesz-Fischer Theorem. There are many statements of the Riesz-Fischer (R-F) Theorem in the literature, some are equivalent, some are consequences of the original versions. The one referred to in this paper is the R-F Theorem which establishes that the spaces $l^2$ is the canonical model space.
\end{abstract}

\noindent \textbf{Keywords: }Bicomplex numbers, Bicomplex algebra, Generalized Hilbert spaces, Riesz-Fischer Theorem.\\
\noindent \textbf{AMS [2010]: }Primary 16D10; Secondary 30G35, 46C05, 46C50.

\newpage

\section{Introduction}
Hilbert spaces over the field of complex numbers are indispensable
for mathematical structure of quantum mechanics \cite{JVN} which in
turn play a great role in molecular, atomic and subatomic
phenomena. The work towards the generalization of quantum
mechanics to bicomplex number system have been recently a topic in
different quantum mechanical models \cite{CV, CVM, BK, Rochon2, Rochon3}.
More specifically, in \cite{GMR2, GMR} the authors made an in depth study of bicomplex
Hilbert spaces and operators acting on them. After obtaining reasonable
results responsible for investigations on finite and infinite
dimensional bicomplex Hilbert spaces and applications to quantum
mechanics \cite{GMR3, GMR4, GeRo, MMR}, they in \cite{GMR} asked for extension of Riesz-Fischer
Theorem and Spectral Theorem on infinite dimensional Hilbert spaces. Recently,
the bicomplex analogue of the Spectral Decomposition Theorem was proven using bicomplex eigenvalues \cite{RKC2}. 
In this paper, we obtain a bicomplex analogue of the Riesz-Fischer
Theorem \cite{Hansen, Horvath} on infinite dimensional Hilbert spaces. Our proof of R-F Theorem
is essentially different from its complex Hilbert space analogue in the sense that we do not make use of the so called Parseval's
identity as done in general Hilbert spaces over $\mR$ or $\mC$. To support our results, we prove A Best Approximation Theorem and we show that the bicomplex analogue of $l^2$, the space of all (real, complex or bicomplex) sequences $\{ w_l \}$ such that
$\sum_{l=1}^{\infty}|w_l|^{2}<\infty$, is a bicomplex Hilbert space. As for the standard quantum mechanics, this specific result is fundamental to understand the space where live the wave functions of the bicomplex Quantum
Harmonic Oscillator \cite{GMR3, Rochon2, Rochon3}.

\section{Preliminaries}

This section first summarizes a number of known results
on the algebra of bicomplex numbers, which will be needed
in this paper.  Much more details as well as proofs can
be found in~\cite{Price, Rochon1, Rochon2, Rochon3}.
Basic definitions related to bicomplex modules and scalar
products are also formulated as in~\cite{GMR2, Rochon3}, but
here we make no restrictions to finite dimensions following definitions of \cite{GMR}.


\subsection{Bicomplex Numbers}\label{Bicomplex Numbers}

\subsubsection{Definition}\label{Definition of bicomplex numbers}
The set $\M(2)$ of \emph{bicomplex numbers} is defined as
\begin{align}
\M(2):=\{ w=z_1+z_2\mathbf{i_2}~|~z_1,z_2\in\mathbb{C}(\mathbf{i_1}) \},
\label{2.1}
\end{align}
where $\ii$ and $\iii$ are independent imaginary units such that
$\iic=-1=\iiic$.  The product of $\ii$ and $\iii$ defines
a hyperbolic unit $\jj$ such that $\mathbf{j}^2=1$.
The product of all units is commutative and satisfies
\begin{equation*}
\ii\iii=\jj, \qquad \ii\jj=-\iii,
\qquad \iii\jj=-\ii. \label{2.2}
\end{equation*}
With the addition and multiplication of two
bicomplex numbers defined in the obvious way,
the set $\M(2)$ makes up a commutative ring.
They are a particular case of the so-called \textit{Multicomplex Numbers} (denoted $\M(n)$) \cite{Price, GR} and \cite{Vaijac}. In fact, bicomplex numbers $$\M(2)\cong {\rm Cl}_{\Bbb{C}}(1,0) \cong {\rm Cl}_{\Bbb{C}}(0,1)$$
are unique among the complex Clifford algebras (see \cite{BDS,DSS} and \cite{Ryan})
in the sense that this set form a commutative, but not division algebra.

Three important subsets of $\M(2)$ can be
specified as
\begin{align*}
\mathbb{C}(\ik{k}) &:= \{ x+y\ik{k}~|~x,y\in\mathbb{R} \},
\qquad k=1,2 ;\label{2.3}\\
\mathbb{D} &:= \{ x+y\jj~|~x,y\in\mathbb{R} \} .
\end{align*}
Each of the sets $\mathbb{C}(\ik{k})$ is isomorphic
to the field of complex numbers, while $\mathbb{D}$ is
the set of so-called \emph{hyperbolic numbers}, also called duplex numbers (see, e.g. \cite {Sob}, \cite {Rochon1}).

\subsubsection{Conjugation and Moduli}\label{Bicomplex conjugation}

Three kinds of conjugation can be defined on
bicomplex numbers. With $w$ specified as in~\eqref{2.1}
and the bar ($\,\bar{\mbox{}}\,$) denoting complex
conjugation in $\mathbb{C}(\mathbf{i_1})$,
we define
\begin{equation*}
w^{\dag_1}:=\bar{z}_1+\bar{z}_2\mathbf{i_2},\label{2.5}
\qquad w^{\dag_2}:=z_1-z_2\mathbf{i_2},
\qquad w^{\dag_3}:=\bar{z}_1-\bar{z}_2\mathbf{i_2} .
\end{equation*}
It is easy to check that each conjugation has the following
properties:
\begin{equation*}
(s+t)^{\dag_k}=s^{\dag_k}+t^{\dag_k},
\qquad \left(s^{\dag_k} \right)^{\dag_k}=s,
\qquad (s\cdot t)^{\dag_k}=s^{\dag_k}\cdot t^{\dag_k} .
\label{2.6}
\end{equation*}
Here $s,t\in\M(2)$ and $k=1,2,3$.

With each kind of conjugation, one can define a specific
bicomplex modulus as
\begin{align*}
|w|_\ii^2&:=w\cdot w^{\dag_2}=z_1^2+z_2^2~\in\C(\ii),\label{2.7a}\\
|w|_\iii^2&:=w\cdot w^{\dag_1}=\left(|z_1|^2-|z_2|^2\right)
+ 2 \, \textrm{Re}(z_1\bar{z}_2)\iii~\in\C(\iii),\\
|w|_\jj^2&:=w\cdot w^{\dag_3}=\left(|z_1|^2+|z_2|^2\right)
- 2 \, \textrm{Im}(z_1\bar{z}_2)\jj~\in\D.
\end{align*}
It can be shown that $|s\cdot t|_k^2=|s|_k^2\cdot|t|_k^2$,
where $k=\ii,\iii$ or $\jj$.

In this paper we will often use the Euclidean $\R^4$-norm
defined as
\begin{equation*}
|w|:=\sqrt{|z_1|^2+|z_2|^2}=\sqrt{\textrm{Re}(|w|_\jj^2)} \; .
\label{2.8}
\end{equation*}
Clearly, this norm maps $\M(2)$ into $\R$.  We have $|w|\geq0$,
and $|w|=0$ if and only if $w=0$. Moreover~\cite{Rochon1},
for all $s,t\in\M(2)$,
\begin{equation*}
|s+t|\leq|s|+|t|, \qquad |s\cdot t|\leq \sqrt{2} \, |s|\cdot|t|.
\label{2.9}
\end{equation*}

\subsubsection{Idempotent Basis}\label{Idempotant basis}

The operations of the bicomplex algebra is considerably simplified by
the introduction of two bicomplex numbers $\ee$
and $\eee$ defined as
\begin{equation*}
\ee:=\frac{1+\jj}{2},\qquad\eee:=\frac{1-\jj}{2}.\label{2.10}
\end{equation*}
In fact $\ee$ and $\eee$ are hyperbolic numbers.
They make up the so-called \emph{idempotent basis}
of the bicomplex numbers. One easily checks that ($k=1,2$)
\begin{equation}
\mathbf{e}_{\mathbf{1}}^2=\ee,
\quad \mathbf{e}_{\mathbf{2}}^2=\eee,
\quad \ee+\eee=1,
\quad \mathbf{e}_{\mathbf{k}}^{\dag_3}=\e{k} ,
\quad \ee\eee=0 . \label{2.11}
\end{equation}

Any bicomplex number $w$ can be written uniquely as
\begin{equation}
w = z_1+z_2\iii = z_\hh \ee + z_\hhh \eee , \label{2.12}
\end{equation}
where
\begin{equation*}
z_\hh= z_1-z_2\ii \quad \mbox{and}
\quad z_\hhh= z_1+z_2\ii \label{2.12a}
\end{equation*}
both belong to $\mathbb{C}(\ii)$.  Note that
\begin{equation*}
|w| = \frac{1}{\sqrt{2}}
\sqrt{|z_\hh |^2 + |z_\hhh |^2} \, . \label{norm7}
\end{equation*}
The caret notation
($\hh$ and $\hhh$) will be used systematically in
connection with idempotent decompositions, with the
purpose of easily distinguishing different types
of indices.  As a consequence of~\eqref{2.11}
and~\eqref{2.12}, one can check that if
$\sqrt[n]{z_\hh}$ is an $n$th root of $z_\hh$
and $\sqrt[n]{z_\hhh}$ is an $n$th root of $z_\hhh$,
then $\sqrt[n]{z_\hh} \, \ee + \sqrt[n]{z_\hhh} \, \eee$
is an $n$th root of $w$.

The uniqueness of the idempotent decomposition
allows the introduction of two projection operators as
\begin{align*}
P_1: w \in\M(2)&\mapsto z_\hh \in\C(\ii),\label{2.14}\\
P_2: w \in\M(2)&\mapsto z_\hhh \in\C(\ii).
\end{align*}
The $P_k$ ($k = 1, 2$) satisfies
\begin{equation*}
[P_k]^2=P_k, \qquad P_1\ee+P_2\eee=\mathbf{Id}, \label{2.16}
\end{equation*}
and, for $s,t\in\M(2)$,
\begin{equation*}
P_k(s+t)=P_k(s)+P_k(t),
\qquad P_k(s\cdot t)=P_k(s)\cdot P_k(t) .\label{2.17}
\end{equation*}

The product of two bicomplex numbers $w$ and $w'$
can be written in the idempotent basis as
\begin{align*}
w \cdot w' = (z_\hh \ee + z_\hhh \eee)
\cdot (z'_\hh \ee + z'_\hhh \eee)
= z_\hh z'_\hh \ee + z_\hhh z'_\hhh \eee .\label{2.20}
\end{align*}
Since 1 is uniquely decomposed as $\ee + \eee$,
we can see that $w \cdot w' = 1$ if and only if
$z_\hh z'_\hh = 1 = z_\hhh z'_\hhh$.  Thus $w$ has an inverse
if and only if $z_\hh \neq 0 \neq z_\hhh$, and the
inverse $w^{-1}$ is then equal to
$(z_\hh)^{-1} \ee + (z_\hhh)^{-1} \eee$.  A nonzero $w$ that
does not have an inverse has the property that
either $z_\hh = 0$ or $z_\hhh = 0$, and such a $w$ is
a divisor of zero.  Zero divisors make up the
so-called \emph{null cone} $\nc$.  That terminology comes
from the fact that when $w$ is written as in~\eqref{2.1},
zero divisors are such that $z_1^2 + z_2^2 = 0$.

Any hyperbolic number can be written in the
idempotent basis as $x_\hh \ee + x_\hhh \eee$, with
$x_\hh$ and $x_\hhh$ in~$\R$.  We define the set~$\D_+$
of positive hyperbolic numbers as
\begin{equation*}
\D_+:= \{ x_\hh \ee + x_\hhh \eee ~|~ x_\hh, x_\hhh \geq 0 \}.
\label{2.21}
\end{equation*}
Since $w^{\dag_3} = \bar{z}_\hh \ee + \bar{z}_\hhh \eee$,
it is clear that $w \cdot w^{\dag_3} \in \D_+$ for any
$w$ in $\M(2)$.

\subsection{$\M(2)$-Module and Scalar Product}\label{Module}

The set of bicomplex numbers is a commutative ring.
Just like vector spaces are defined over fields,
modules are defined over rings. A module~$M$ defined
over the ring of bicomplex numbers is called an
$\M(2)$-\emph{module}~\cite{Rochon3, GMR2, GMR}.

Let $M$ be an $\M(2)$-module. For $k=1, 2$, we define $V_k$
as the set of all elements of the form $\e{k} \ket{\psi}$,
with $\ket{\psi} \in M$.  Succinctly, $V_1:=\eo M$
and $V_2:=\et M$. In fact, $V_k$ is a vector space over $\C(\ii)$ and any element $\ket{v_k} \in V_k$ satisfies
$\ket{v_k} = \e{k} \ket{v_k}$ for $k=1,2$.
For arbitrary $\M(2)$-modules, vector spaces $V_1$ and $V_2$
bear no structural similarities.  For more specific modules,
however, they may share structure.  It was shown in~\cite{GMR2}
that if~$M$ is a finite-dimensional free $\M(2)$-module, then
$V_1$ and~$V_2$ have the same dimension.

For any $\ket{\psi}\in M$, there exist a unique decomposition
\begin{align}
\ket{\psi} = \ket{v_1}
+ \ket{v_2}, \label{2.31}
\end{align}
where $v_k\in V_k$, $k=1,2$.

It will be useful to rewrite \eqref{2.31} as
\begin{align*}
\ket{\psi} = \ket{\psi_\mo}
+ \ket{\psi_\mt} , 
\end{align*}
where
\begin{align*}
\ket{\psi_\mo} := \eo\ket{\psi}  && \text{and} && \ket{\psi_\mt} := \et\ket{\psi} .
\end{align*}

In fact, the $\M(2)$-module $M$ can be viewed as a vector space $M'$
over $\mC(\bo)$, and $M'=V_1\oplus V_2.$ From a set-theoretical point of view, $M$ and $M'$ are
identical.  In this sense we can say, perhaps improperly,
that the \textbf{module} $M$ can be decomposed into the
direct sum of two vector spaces over $\mC(\bo)$, i.e.\
$M=V_1\oplus V_2.$

\subsubsection{Bicomplex Scalar Product}\label{bicomplex sc}

A \emph{bicomplex scalar product} maps two arbitrary kets
$\ket{\psi}$ and $\ket{\phi}$ into a bicomplex number
$(\ket{\psi}, \ket{\phi})$, so that the following
always holds ($s \in \M(2)$):
\begin{enumerate}
\item $(\ket{\psi}, \ket{\phi} + \ket{\chi})
=(\ket{\psi}, \ket{\phi}) + (\ket{\psi}, \ket{\chi})$;
\item $(\ket{\psi}, s \ket{\phi})
= s (\ket{\psi},\ket{\phi})$;
\item $(\ket{\psi}, \ket{\phi})
= (\ket{\phi}, \ket{\psi})^{\dagger_3}$;
\item $(\ket{\psi}, \ket{\psi})
=0~\Leftrightarrow~\ket{\psi}=0$.
\end{enumerate}
The bicomplex scalar product was defined in~\cite{Rochon3}
where, as in this paper, the physicists' convention is used
for the order of elements in the product.

Property $3$ implies that $(\ket{\psi}, \ket{\psi})\in\D$,
while properties 2 and 3 together imply that
$(s \ket{\psi}, \ket{\phi}) = s^{\dagger_3}
(\ket{\psi},\ket{\phi})$. However, in this work we will also require the
bicomplex scalar product $\(\cdot,\cdot\)$ to be \textit{hyperbolic
positive}, i.e.
\begin{align*}
(\ket{\psi},\ket{\psi})\in\mathbb{D}_{+},\mbox{
}\forall\ket{\psi}\in M. 
\end{align*}
This is a necessary condition if we want to recover the standard quantum mechanics from the bicomplex one (see \cite{GMR3}).
\begin{definition}
Let $M$ be a $\mathbb{T}$-module and let $(\cdot,\cdot)$
be a bicomplex scalar product defined on $M$. The space
$\{M, (\cdot,\cdot)\}$ is called a $\M(2)$-inner product
space, or bicomplex pre-Hilbert space.  When no confusion
arises, $\{M, (\cdot,\cdot)\}$ will simply be denoted by~$M$.
\end{definition}

In this work, we will sometimes use the Dirac notation
\begin{align*}
(\ket{\psi},\ket{\phi})=\braket{\psi}{\phi} 
\end{align*}
for the scalar product. The one-to-one correspondence between \emph{bra} $\bra{\cdot}$ and
\emph{ket} $\ket{\cdot}$ can be established from the Bicomplex Riesz Representation Theorem \cite[Th. 3.7]{GMR}.
As in \cite{GeRo}, subindices will be used inside the ket notation. In fact, this is simply a convenient way to deal with the Dirac notation in $V_1$ and $V_2$.
Note that the following projection of a bicomplex scalar product:
\begin{equation*}
(\cdot,\cdot)_{\widehat{k}}:=P_k((\cdot,\cdot)):M\times M\longrightarrow \mC(\bo)
\end{equation*}
is a \textbf{standard scalar product} on $V_k$, for $k=1,2$. One easily show \cite{GMR} that
\begin{align}
(\ket{\psi}, \ket{\phi})
&= \ee\P{1}{(\ket{\psi_\mo}, \ket{\phi_\mo})}
+ \eee\P{2}{(\ket{\psi_\mt}, \ket{\phi_\mt})} \notag\\
&=\ee\scalarmath{\ket{\psi_\mo}}{\ket{\phi_\mo}}_\hh+\eee\scalarmath{\ket{\psi_\mt}}{\ket{\phi_\mt}}_\hhh. \notag\\
&=\ee\braket{\psi_\mo}{\phi_\mo}_\hh+\eee\braket{\psi_\mt}{\phi_\mt}_\hhh\label{2.36}.
\end{align}

We point out that a bicomplex scalar product is
\textbf{completely characterized} by the two standard
scalar products $\scalarmath{\cdot}{\cdot}_{\widehat{k}}$ on $V_k$.
In fact, if $\scalarmath{\cdot}{\cdot}_{\widehat{k}}$
is an arbitrary scalar product on $V_k$, for $k=1,2$,
then $\scalarmath{\cdot}{\cdot}$ defined as in \eqref{2.36}
is a bicomplex scalar product on $M$.

From this scalar product, we can define a \textbf{norm}
on the vector space $M'$:
\begin{align}
\big{|}\big{|}\ket{\phi}\big{|}\big{|}
&:= \frac{1}{\sqrt{2}}
\sqrt{\scalarmath{\ket{\phi_\mo}} {\ket{\phi_\mo}}_{\widehat{1}}
+ \scalarmath{\ket{\phi_\mt}}{\ket{\phi_\mt}}_{\widehat{2}}} \notag\\
&=\frac{1}{\sqrt{2}} \sqrt{ \big{|}\ket{\phi_\mo}\big{|}^{2}_{1}
+ \big{|}\ket{\phi_\mt}\big{|}^{2}_{2}} \, .
\label{T-norm}
\end{align}
Here we wrote
\begin{equation*}
\big{|}\ket{\phi_\mk} \big{|}_{k}
= \sqrt{\scalarmath{\ket{\phi_\mk}}
{\ket{\phi_\mk}}_{\widehat{k}}} \, ,
\end{equation*}
where $|\cdot|_k$ is the natural scalar-product-induced norm on~$V_k$.
Moreover,
\begin{equation}
\big{|}\big{|}\ket{\phi}\big{|}\big{|}
= \frac{1}{\sqrt{2}}
\sqrt{\scalarmath{\ket{\phi_\mo}} {\ket{\phi_\mo}}_{\widehat{1}}
+ \scalarmath{\ket{\phi_\mt}}{\ket{\phi_\mt}}_{\widehat{2}}} \notag
= \big{|} \sqrt{\scalarmath{\ket{\phi}}{\ket{\phi}}} \big{|}.
\label{T-norma}
\end{equation}

\begin{definition}
Let $M$ be an $\M(2)$-module and let $M'$ be the associated
vector space. We say that $\|\cdot\|:M\longrightarrow \mathbb{R}$
is a \textbf{$\M(2)$-norm} on $M$ if the following holds:

\smallskip\noindent
1. $\|\cdot\|:M'\longrightarrow \mathbb{R}$ is a norm;\\
2. $\big{\|}w\cdot \ket{\psi}\big{\|}\leq \sqrt{2}
\big{|}w\big{|}\cdot\big{\|}\ket{\psi}\big{\|}$,
$\forall w\in\T$, $\forall \ket{\psi}\in M$.
\label{norm}
\end{definition}
\noindent A $\M(2)$-module with a \textbf{$\M(2)$-norm} is called a
\textbf{normed $\M(2)$-module}. It is easy to check that $\|\cdot\|$ in \eqref{T-norm} is a \textbf{$\M(2)$-norm} on $M$
and that the $\M(2)$-module $M$ is \textbf{complete} with respect
to the following metric on $M$:
\begin{equation*}
d(\ket{\phi},\ket{\psi})=\big{|}\big{|}\ket{\phi}-\ket{\psi}\big{|}\big{|}
\end{equation*}
if and only if $V_1$ and $V_2$ are complete (see \cite{GMR}).
\begin{definition}
A bicomplex Hilbert space is a $\M(2)$-inner product space $M$ which is complete with respect to the induced $\M(2)$-norm \eqref{T-norm}.
\label{Hilbert}
\end{definition}

\section{Main results}

Throughout the text, by a \textbf{bicomplex Hilbert space} we
shall mean an infinite dimensional bicomplex Hilbert space.
A normed $\M(2)$-module with a Schauder $\M(2)$-basis is called a
\textbf{countable $\M(2)$-module}.

\begin{definition}
A bicomplex Hilbert space $M$ is said to be \textit{separable by a basis} if it has
a Schauder $\M(2)$-basis.
\end{definition}

We note that by Theorem 3.10 in \cite{GMR}, any Schauder $\M(2)$-basis of $M$ can be
orthonormalized.

\begin{remark}
A topological space $S$ is called \textit{separable} if it admits a countable dense subset $W$.
\end{remark}

\begin{proposition}
Let $\braket{\cdot}{\cdot}$ be a bicomplex inner product in the bicomplex Hilbert space $M$ and let $||\cdot||$ be the induced norm.
If the sequences $\{\ket{\psi_n}\}$ and $\{\ket{\phi_n}\}$ in $M$ converge to  $\{\ket{\psi}\}$ and
$\{\ket{\phi}\}$ respectively, then the sequence of inner products $\{ \braket{{\psi_n}}{{\phi_n}}\}$
converges to $\braket{{\psi}}{{\phi}}$.
\label{SCALIM}
\end{proposition}
\begin{proof}
First observe that:
$\braket{{\psi_n}}{{\phi_n}}-\braket{{\psi}}{{\phi}}$
\begin{eqnarray*}
&=& \braket{{\psi_n}}{{\phi_n}}-\braket{{\psi}}{{\phi_n}}+\braket{{\psi}}{{\phi_n}}-\braket{{\psi}}{{\phi}}\\
&=& \braket{\psi_n - \psi}{\phi_n} + \braket{\psi}{\phi_n - \phi}\\
&=& \braket{{\psi_n}-{\psi}}{{\phi_n}-{\phi}}
+\braket{{\psi_n}-{\psi}}{{\phi}}
+ \braket{{\psi}}{{\phi_n}-{\phi}}.
\end{eqnarray*}
From this we get by the \textbf{bicomplex Schwarz inequality} (\cite{GMR}, Theorem 3.8):
$\big{|}\braket{{\psi_n}}{{\phi_n}}-\braket{{\psi}}{{\phi}}\big{|}$
\begin{eqnarray*}
&=& \big{|}\braket{{\psi_n}-{\psi}}{{\phi_n}-{\phi}}
+\braket{{\psi_n}-{\psi}}{{\phi}}
+ \braket{{\psi}}{{\phi_n}-{\phi}}\big{|}\\
&\leq& \big{|}\braket{{\psi_n}-{\psi}}{{\phi_n}-{\phi}}\big{|}
+\big{|}\braket{{\psi_n}-{\psi}}{{\phi}}\big{|}
+\big{|}\braket{{\psi}}{{\phi_n}-{\phi}}\big{|}\\
&\leq& \big{[}\sqrt{2}\big{|}\big{|}\ket{\psi_n}-\ket{\psi}\big{|}\big{|}\cdot\big{|}\big{|}\ket{\phi_n}-\ket{\phi}\big{|}\big{|}
+\sqrt{2}\big{|}\big{|}\ket{\psi_n}-\ket{\psi}\big{|}\big{|}\cdot\big{|}\big{|}\ket{\phi}\big{|}\big{|}\\
& &+\sqrt{2}\big{|}\big{|}\ket{\psi}\big{|}\big{|}\cdot\big{|}\big{|}\ket{\phi_n}-\ket{\phi}\big{|}\big{|}\big{]}.
\end{eqnarray*}
The proposition now follows easily.
\end{proof}

\begin{theorem}[Best Approximation Theorem]
Let  $\{ \ket{\psi_n}\}$ be an arbitrary orthonormal sequence in the bicomplex Hilbert space $M=H_1\oplus H_2$, and let $\alpha_1,\ldots,\alpha_n$ be a set of bicomplex numbers. Then for all $\ket{\psi}\in M$,
$$\big{|}\big{|}\ket{\psi}-\sum_{l=0}^{n}\alpha_l \ket{\psi_l}\big{|}\big{|}\geq \big{|}\big{|}\ket{\psi}-\sum_{l=0}^{n}\braket{{\psi_l}}{{\psi}} \ket{\psi_l}\big{|}\big{|}.$$
\label{BEST}
\end{theorem}
\begin{proof}
By definition of the bicomplex inner product, the set $\{\ket{\psi_{n\mk}}\}$ is also an arbitrary orthonormal sequence in the Hilbert space $H_k$ for $k=1,2$. Therefore, using the classical Best Approximation Theorem (see \cite{Hansen}, P.61) on the Hilbert spaces $H_1$ and $H_2$, we obtain for $k=1,2$:
\begin{equation*}
\big{|}\ket{\psi_\mk}-\sum_{l=0}^{n}P_k(\alpha_l) \ket{\psi_{l\mk}}\big{|}_{k}\geq
\big{|}\ket{\psi_\mk}-\sum_{l=0}^{n}\braket{{\psi_{l\mk}}}{{\psi_\mk}}_{\widehat{k}} \ket{\psi_{l\mk}}\big{|}_{k}.
\end{equation*}
Hence, by definition of the $\M(2)$-norm, we have that
\begin{eqnarray*}
\big{|}\big{|}\ket{\psi}-\sum_{l=0}^{n}\alpha_l \ket{\psi_l}\big{|}\big{|} &=& \frac{1}{\sqrt{2}} \sqrt{\sum_{k=1}^{2} \big{|}\ket{\psi_\mk}-\sum_{l=0}^{n}P_k(\alpha_l) \ket{\psi_{l\mk}}\big{|}_{k}^{2}}\\
&\geq& \frac{1}{\sqrt{2}} \sqrt{\sum_{k=1}^{2} \big{|}\ket{\psi_\mk}-\sum_{l=0}^{n}
\braket{{\psi_{l\mk}}}{{\psi_\mk}}_{\widehat{k}}  \ket{\psi_{l\mk}}\big{|}_{k}^{2}}\\
&=& \big{|}\big{|}\ket{\psi}-\sum_{l=0}^{n}\braket{{\psi_l}}{{\psi}} \ket{\psi_l}\big{|}\big{|}.
\end{eqnarray*}
\end{proof}

An important consequence of the Best Approximation Theorem is that an orthonormal basis for a dense subspace
of a bicomplex Hilbert space is actually an orthonormal basis in the full bicomplex Hilbert space. This is very useful
result for the construction of specific orthonormal basis in separable Hilbert spaces. The precise result is as follows.

\begin{theorem}
Let $N$ be a dense subspace of the bicomplex Hilbert space $M$, and assume that $\{ \ket{m_l} \}$ is an orthonormal
Schauder $\M(2)$-basis for $N$. Then $\{ \ket{m_l} \}$ is also an orthonormal Schauder $\M(2)$-basis for $M$.
\label{DENSE}
\end{theorem}
\begin{proof}
Since $\{ \ket{m_l} \}$ is a Schauder $\M(2)$-basis for $N$, any $\ket{\psi}\in N$ admits a unique expansion as
an infinite series $\ket{\psi}=\sum_{l=1}^{\infty} \alpha_l \ket{m_l}$. In fact,
$$\ket{\psi}=\sum_{l=1}^{\infty} \braket{{m_l}}{{\psi}} \ket{m_l}.$$
This follows by Proposition \ref{SCALIM} and the short computation
$$\braket{{m_l}}{{\psi}}=\braket{{m_l}}{\lim_{n\rightarrow \infty}\sum_{k=1}^{n}\alpha_k {m_k}}
=\lim_{n\rightarrow \infty}\braket{{m_l}}{\sum_{k=1}^{n}\alpha_k {m_k}}
=\alpha_l,$$
valid for all $l\in\mathbb{N}$. Now, to complete the proof, let us prove that any ket $\ket{\phi}\in M$ admits the same expansion form:
\begin{equation}
\ket{\phi}=\sum_{l=1}^{\infty} \braket{{m_l}}{{\phi}} \ket{m_l}.
\label{Unique}
\end{equation}
To prove this assertion, let an arbitrary $\epsilon>0$ be given. Since, $N$ is dense in $M$,
we can choose $\ket{\psi}\in N$, such that $\big{|}\big{|}\ket{\phi}-\ket{\psi} \big{|}\big{|}<\frac{\epsilon}{2}$.
Now write $\ket{\psi}=\sum_{l=1}^{\infty} \braket{{m_l}}{{\psi}} \ket{m_l}$, and choose $n_0\in\mathbb{N}$
such that
$$n\geq n_0 \Rightarrow \big{|}\big{|} \ket{\psi}-\sum_{l=1}^{n} \braket{{m_l}}{{\psi}} \ket{m_l} \big{|}\big{|}<\frac{\epsilon}{2}.$$
By the Best Approximation Theorem, we then get for all $n\geq n_0$,
\begin{eqnarray*}
\big{|}\big{|} \ket{\phi}-\sum_{l=1}^{n} \braket{{m_l}}{{\phi}} \ket{m_l} \big{|}\big{|} &\leq&
\big{|}\big{|} \ket{\phi}-\sum_{l=1}^{n} \braket{{m_l}}{{\psi}} \ket{m_l} \big{|}\big{|} \\
&\leq& \big{|}\big{|}\ket{\phi}-\ket{\psi} \big{|}\big{|}+\big{|}\big{|}\ket{\psi}- \sum_{l=1}^{n} \braket{{m_l}}{{\psi}} \ket{m_l}\big{|}\big{|}\\
&\leq& \frac{\epsilon}{2}+\frac{\epsilon}{2}.
\end{eqnarray*}
Hence,
\begin{equation*}
\ket{\phi}=\lim_{n\rightarrow\infty} \sum_{l=1}^{n} \braket{{m_l}}{{\phi}} \ket{m_l}=\sum_{l=1}^{\infty} \braket{{m_l}}{{\phi}} \ket{m_l}.
\end{equation*}
This prove that $\{ \ket{m_l} \}$ is an orthonormal Schauder $\M(2)$-basis for $M$.
\end{proof}

The next result shows that all separable bicomplex Hilbert spaces are separable by a basis.

\begin{lemma}
Every separable bicomplex Hilbert space $M$ has an orthonormal Schauder $\M(2)$-basis.
\label{SHS03}
\end{lemma}
\begin{proof}
By the definition of separability, $M$ contains a countable, dense subset $W$ of kets in $M$.
Consider the linear subspace $U$ in $M$ consisting of all finite bicomplex linear combinations of kets
in $W$ - the \textit{bicomplex linear span} of $W$. Clearly, $U$ is a dense sub-$\M(2)$-module in $M$.
By the construction of $U$ we can eliminate kets from the countable set $W$ one after the other
to get a (bicomplex) linearly independent set $\{ \ket{\phi_n} \}$ (finite, or countable) of kets in $U$ that
spans $U$. However, a sub-$\M(2)$-module $U$ in $M$ of finite dimension is a complete space, thus a closed set in
$M$, and then $U=\bar{U}=M$ a contradiction with our hypothesis. Therefore, the set $\{ \ket{\phi_n} \}$ is
a countable (bicomplex) linearly independent set of kets in $U$. Now, since no $\ket{\phi_n}$
(and thus no $\braket{{\phi_n}}{{\phi_n}}$) can belongs to the null cone, the
classical Gram-Schmidt process can be applied (see \cite{GMR2}, P.14). Hence, we can turn
the sequence $\{ \ket{\phi_n} \}$ into an orthonormal sequence $\{ \ket{\psi_n} \}$ with
the property that for all $n\in\mN$,
$$\mbox{span} \{ \ket{\phi_n} \}_{l=1}^{n}=\mbox{span} \{ \ket{\psi_l} \}_{l=1}^{n}$$
Since $\{ \ket{\psi_l} \}$ is orthonormal, we can use $\{ \ket{\psi_l} \}$  as a Schauder
$\M(2)$-basis to generate a linear subspace $N$ in $M$ (for the unicity, see the proof of Theorem \ref{DENSE}).
Then $N$ is a dense sub-$\M(2)$-module in $M$, since $U$ is a dense sub-$\M(2)$-module in $N$.
The latter follows since any ket $\ket{\psi}\in N$ can be expanded into a series
$\ket{\psi}=\sum_{l=1}^{\infty}\alpha_l  \ket{\psi_l}$, showing that
$\ket{\psi}=\lim_{n\rightarrow \infty}\sum_{l=1}^{n}\alpha_l  \ket{\psi_l}$,
and hence that $\ket{\psi}$ is the limit of a sequence of kets in $U$.

By construction, $\{ \ket{\psi_l} \}$ is an orthonormal Schauder
$\M(2)$-basis for $N$ and hence by Theorem \ref{DENSE} also for $M$.
\end{proof}

\begin{theorem}
If $M$ is a separable bicomplex Hilbert space, then $H_k$ ($k=1,2$) is an infinite dimensional separable complex
Hilbert space.
\label{SHS}
\end{theorem}
\begin{proof}
From Lemma \ref{SHS03}, $M=H_1\oplus H_2$ has an orthonormal Schauder $\M(2)$-basis  $\{ \ket{\psi_l} \}$. It is easy to see that
$\{ \ket{\psi_{l\mk}} \}$ is also an orthonormal Schauder basis for $H_k$ ($k=1,2$). Hence, $H_k$ ($k=1,2$) is
separable by a basis. Now, from Theorem 3.3.6. in \cite{Hansen}, $H_k$ ($k=1,2$) is an infinite dimensional separable complex
Hilbert space.
\end{proof}

\begin{definition}
Denote by $l^2_2$, the space of all (real, complex or bicomplex) sequences $\{ w_l \}$ such that
$$
\sum_{l=1}^{\infty}|w_l|^{2}<\infty.
$$
\end{definition}

The bicomplex $l^2_2$ space is clearly an $\M(2)$-module. The norm of the associated vector space $({l^2_2})'$ over $\mathbb{C}(\ik{1})$ is
defined by
\begin{equation*}
||\{ w_l \}||_2=\Big( \sum_{l=1}^{\infty}|w_l|^{2} \Big)^{\frac{1}{2}}.
\end{equation*}

\begin{theorem}
$l^2_2$ is a bicompex Hilbert space.
\end{theorem}
\begin{proof}
Let us prove that $({l^2_2})'=(\eo l^2) \oplus (\et l^2)$. This comes automatically from the fact that
any bicomplex sequence $\{ w_l \}$ can be decomposed as
the following sum of two sequences in $\mathbb{C}(\ik{1})$:
$$\{ w_l \}=\eo\{ z_{1l}-z_{2l}\ik{1} \}+\et\{z_{1l}+z_{2l}\ik{1} \}.$$
To complete the proof, we need to verify that the norm $||\cdot||_2$
coincides with the induced $\M(2)$-norm of the bicomplex Hilbert space $(\eo l^2) \oplus (\et l^2)$.
Let $||\cdot||$ be the induced $\M(2)$-norm of the bicomplex Hilbert space $(\eo l^2) \oplus (\et l^2)$. Thus
$$\big{|}\big{|}\{ w_l \}\big{|}\big{|}=
\frac{1}{\sqrt{2}} \sqrt{ \big{|}\{ z_{1l}-z_{2l}\ik{1} \}\big{|}^{2}_{1}
+ \big{|}\{z_{1l}+z_{2l}\ik{1} \}\big{|}^{2}_{2}}$$
where $\big{|}\cdot\big{|}_{1}=\big{|}\cdot\big{|}_{2}$ is the classical norm on $l^2$.
Hence,
\begin{align*}
\big{|}\big{|}\{ w_l \}\big{|}\big{|}
&= \frac{1}{\sqrt{2}} \sqrt{ \big{|}\{ z_{1l}-z_{2l}\ik{1} \}\big{|}^{2}_{1}
+ \big{|}\{z_{1l}+z_{2l}\ik{1} \}\big{|}^{2}_{1}}\\
&= \frac{1}{\sqrt{2}} \sqrt{ \sum_{l=1}^{\infty}|z_{1l}-z_{2l}\ik{1}|^{2}
+\sum_{l=1}^{\infty}|z_{1l}+z_{2l}\ik{1}|^{2}}\\
&=\sqrt{\sum_{l=1}^{\infty}\frac{[|z_{1l}-z_{2l}\ik{1}|^{2}
+|z_{1l}+z_{2l}\ik{1}|^{2}]}{2}}\\
&=||\{ w_l \}||_2.
\end{align*}

\end{proof}

We are now ready for the proof of the main result on the structure of infinite dimensional, separable
bicomplex Hilbert space. We show that the space of square summable bicomplex sequences $l^2_2$  Define the projection $T_\mk:M\longrightarrow V_k$ as
\begin{equation*}
T_\mk \ket{\phi}
:=  \e{k}T(\ket{\phi}), \;
\forall \ket{\phi} \in M, \; k=1,2. \notag
\end{equation*}

\noindent With this definition we have the following Lemma.

\begin{lemma}
Let $M_1, M_2$ be two $\M(2)$-modules and $T:M_1\rightarrow M_2$ be a bicomplex
linear function. Then
$\forall\ket{\phi}\in M_1$ we have
\begin{equation*}
T_{\mk}(\ket{\phi})=T(\ket{\phi_{\mk}}), \ \
(k=1,2).
\end{equation*}
\label{ISO}
\end{lemma}
\begin{proof}
\begin{align*}
T_{\mk}(\ket{\phi})
&= \ek (T(\ket{\phi}))\\
&= \ek (T(\ket{\phi_{\mo}}+\ket{\phi_{\mt}}))\\
&= T((\ket{\phi_{\mk}})).
\end{align*}
\end{proof}

\begin{theorem}[Riesz-Fischer] Every separable bicomplex Hilbert space $M$ is
isometrically isomorphic to the bicomplex Hilbert space $l^2_2$.
\end{theorem}
\begin{proof}
From Lemma \ref{SHS03}, since $M=H_1\oplus H_2$ is a separable bicomplex Hilbert
space, it has an orthonormal Schauder $\M(2)$-basis: $$\{\ket{m_1}, \dots ,\ket{m_l}, \dots\}.$$
Then each
$\ket{\psi}\in M$ admits a unique decomposition as
$$\ket{\psi}=\sum_{l=1}^{\infty}w_l\ket{m_l}, \ \ w_l\in\M(2).$$

Since the infinite series above converges, by Theorem 3.11 in
\cite{GMR}, the series $\sum_{l=1}^{\infty}\left|w_l\right|^{2}$
converges in $\mR$ and thus $\{w_l\}\in l^2_2$. Now, define a map
$T:M\rightarrow l^2_2$ as
$$T(\ket{\phi})=\{w_l\}_{l=1}^{\infty} \ \ \forall \ket{\phi}\in
M.$$
$T$ is a well defined map: Let $\ket{\phi}, \ \ket{\psi} \in M$ be
such that $\ket{\phi}=\ket{\psi}$. Hence,
$\sum_{l=1}^{\infty}w_l\ket{m_l}
=\sum_{l=1}^{\infty}{w_l}{\prime}\ket{m_l}$ and then by the uniqueness
of the representation we find that $w_l ={w_l}^{\prime}$ for each
$l\in \mN$, which further implies that
$T(\ket{\phi})=T(\ket{\psi})$. Next, we show that $T$ is \textbf{bicomplex} linear. Let $\ket{\phi}, \ \ket{\psi}
\in M$ and $\alpha, \ \beta \in \mathbb {T}$. Then,
\begin{eqnarray*}
T(\alpha \ket{\phi}+\beta \ket{\psi}) &=&
T(\alpha\sum_{l=1}^{\infty}w_l\ket{m_l} +\beta
\sum_{l=1}^{\infty}{w_l}{\prime}\ket{m_l})\\
&=& T(\sum_{l=1}^{\infty}(\alpha w_l)\ket{m_l} +
\sum_{l=1}^{\infty}(\beta {w_l}{\prime})\ket{m_l})\\
&=& T(\sum_{l=1}^{\infty}(\alpha w_l + \beta
{w_l}{\prime})\ket{m_l})\\
&=& \{\alpha w_l +\beta {w_l}^{\prime} \}\\
&=& \alpha \{ w_l\}+\beta \{{w_l}^{\prime} \}\\
&=& \alpha T(\ket{\phi})+\beta T(\ket{\psi}).
\end{eqnarray*}
Now, since $\{ \ket{w_l} \}$ is an orthonormal basis in $M$, by Equation \eqref{Unique} in Theorem \ref{DENSE}, every ket $\ket{\phi}\in M$ admits
the unique expansion
$$\ket{\phi}=\sum_{l=1}^{\infty} \braket{{m_l}}{{\phi}} \ket{m_l}.$$
Hence, $T$ is \textbf{injective}, since $T(\ket{\phi})=\{ \braket{{m_l}}{{\phi}} \}=0$ implies $\braket{{m_l}}{{\phi}}=0$
for all $l\in\mathbb{N}$, and thus $\ket{\phi}=0$. Moreover, $T$ is \textbf{surjective}, since for any element $\{ \alpha_l \}\in l^2_2$,
the series $\ket{\xi}=\sum_{l=1}^{\infty}\alpha_l\ket{m_l}$ is convergent (Theorem 3.11 in \cite{GMR}).
Finally we shall show that $T$ is an isometry.
By Lemma \ref{ISO}, we have that
\begin{eqnarray}
\big{\|}T(\ket{\phi})\big{\|}
&=& \big{|}\sqrt{\scalarmath{T(\ket{\phi})}{T(\ket{\phi})}}\big{|}\notag\\
&=& \big{|}\sqrt {\eo
\scalarmath{T(\ket{\phi_{\mo}})}{T(\ket{\phi_{\mo}})}_{\widehat{1}} +
\et \scalarmath{T(\ket{\phi_{\mt}})}{T(\ket{\phi_{\mt}})}_{\widehat{2}}}\big{|}\notag\\
\label{iso9001}
\end{eqnarray}

By Theorem \ref{SHS}, the classical Riesz-Fischer Theorem can be applied to
$H_k$ where $T:H_k\rightarrow \ek l^2$ for $k=1,2$. Then we find that
\begin{eqnarray*}
\scalarmath{T(\ket{\phi_{\mk}})}{T(\ket{\phi_{\mk}})}_{\widehat{k}} &=&
{\big{|}T(\ket{\phi_{\mk}})\big{|}^{2}_k} \\
&=& { \big{|}\ket{\phi_{\mk}}\big{|}^{2}_k} \\
&=& \braket{{\phi_{\mk}}}{{\phi_{\mk}}}_{\widehat{k}}
\end{eqnarray*}
for $k=1,2$, where $\big{|}\cdot\big{|}_{1}=\big{|}\cdot\big{|}_{2}$ is the classical norm on $l^2$.
Thus, from Equation \eqref{iso9001}, we get that
\begin{equation*}
\big{\|}T(\ket{\phi})\big{\|}=\big{\|}\ket{\phi}\big{\|}.
\end{equation*}
This proves that $T$ is an isometry. Hence $M$ is isometrically
isomorphic to the bicomplex Hilbert space $l^2_2$.
\end{proof}

\section*{Acknowledgment}
DR is grateful to the Natural Sciences and
Engineering Research Council of Canada for financial
support.


\bibliographystyle{amsplain}

\end{document}